\documentclass[12pt, reqno]{amsart}
\usepackage{amsmath, amsthm, amscd, amsfonts, amssymb, graphicx, color}
\usepackage[bookmarksnumbered, colorlinks, plainpages]{hyperref}
\hypersetup{colorlinks=true,linkcolor=red, anchorcolor=green, citecolor=cyan, urlcolor=red, filecolor=magenta, pdftoolbar=true}
\input{mathrsfs.sty}

\newtheorem{theorem}{Theorem}[section]
\newtheorem{lemma}[theorem]{Lemma}
\newtheorem{proposition}[theorem]{Proposition}

\theoremstyle{definition}
\newtheorem{definition}[theorem]{Definition}

\theoremstyle{remark}
\newtheorem{remark}[theorem]{Remark}
\numberwithin{equation}{section}
 
\begin{document}

\title[Left multipliers of reproducing kernel Hilbert $C^*$-modules]{Left multipliers of reproducing kernel Hilbert $C^*$-modules and the Papadakis theorem}

\author[M. Ghaemi, V. M. Manuilov, M. S. Moslehian]{Mostafa Ghaemi$^1$, Vladimir M. Manuilov$^2$, \MakeLowercase{and} Mohammad Sal Moslehian$^3$}

\address{$^1$Department of Pure Mathematics, Ferdowsi University
of Mashhad, P. O. Box 1159, Mashhad 91775, Iran.}
\email{mofigh072@gmail.com}

\address{$^3$Moscow Center for Fundamental and Applied Mathematics, and Department of Mechanics and Mathematics, Moscow State University, Moscow, 119991, Russia.} 
\email{manuilov@mech.math.msu.su}

\address{$^3$Department of Pure Mathematics, Center of Excellence in Analysis on Algebraic Structures (CEAAS), Ferdowsi University of Mashhad, P. O. Box 1159, Mashhad 91775, Iran.}
\email{moslehian@um.ac.ir; moslehian@yahoo.com}

\date{}
\renewcommand{\subjclassname}{\textup{2020} Mathematics Subject Classification}


\begin{abstract}
We give a modified definition of a reproducing kernel Hilbert $C^*$-module (shortly, $RKHC^*M$) without using the condition of self-duality and discuss some related aspects; in particular, an interpolation theorem is presented. We investigate the exterior tensor product of $RKHC^*M$s and find their reproducing kernel. In addition, we deal with left multipliers of $RKHC^*M$s. Under some mild conditions, it is shown that one can make a new $RKHC^*M$ via a left multiplier. Moreover, we introduce the Berezin transform of an operator in the context of $RKHC^*M$s and construct a unital subalgebra of the unital $C^*$-algebra consisting of adjointable maps on an $RKHC^*M$ and show that it is closed with respect to a certain topology. Finally, the Papadakis theorem is extended to the setting of $RKHC^*M$, and in order for the multiplication of two specific functions to be in the Papadakis $RKHC^*M$, some conditions are explored.
\end{abstract}

\maketitle

\section{Introduction}

Hilbert $C^*$-modules are generalization of Hilbert spaces by allowing the inner product to take its values in a $C^*$-algebra instead of the complex numbers. At the same time, they are extensions of $C^*$-algebras. Indeed, a $C^*$-algebra $\mathfrak{A}$ is a Hilbert $\mathfrak{A}$-module if we define $\langle a,b\rangle=a^*b$ $(a,b\in \mathfrak{A})$. For Hilbert $C^*$-modules $\mathcal{E}$ and $\mathcal{F}$, the set of all adjointable maps from $\mathcal{E}$ to $\mathcal{F}$ is denoted by $L(\mathcal{E},\mathcal{F})$, and $L(\mathcal{E})$ stands for the unital $C^*$-algebra $L(\mathcal{E},\mathcal{E})$. We assume that $\mathfrak{A}\otimes_{alg}\mathfrak{B}$ and $\mathfrak{A}\otimes_*\mathfrak{B}$ denote, respectively, the algebraic tensor product and an arbitrary fixed $C^*$-tensor product of the $C^*$-algebras $\mathfrak{A}$ and $\mathfrak{B}$ with the corresponding $C^*$-tensor norm $\|\cdot\|_*$. For more details on the general theory of $C^*$-algebras, the reader is referred to \cite{Mor}. Let $\mathcal{E}_1$ and $\mathcal{E}_2$ be Hilbert $C^*$-modules over $C^*$-algebras $\mathfrak{A}$ and $\mathfrak{B}$, respectively. The algebraic tensor product and the exterior tensor product of $\mathcal{E}_1$ and $\mathcal{E}_2$ are denoted by $\mathcal{E}_1\otimes_{alg}\mathcal{E}_2$ and $\mathcal{E}_1\otimes \mathcal{E}_2$, respectively. Indeed, $\mathcal{E}_1\otimes \mathcal{E}_2$ is a Hilbert $\mathfrak{A}\otimes_*\mathfrak{B}$-module equipped with the following $\mathfrak{A}\otimes_*\mathfrak{B}$-valued inner product \cite{Lance}:
\[
\langle x\otimes y, z\otimes w\rangle=\langle x,z\rangle \otimes \langle y,w\rangle\quad (x,z\in \mathcal{E}_1, y,w\in \mathcal{E}_2).
\]
Let $\mathcal{E}$ be a Hilbert $C^*$-module over a $C^*$-algebra $\mathfrak{A}$. The set of all bounded $\mathfrak{A}$-linear maps from $\mathcal{E}$ to $\mathfrak{A}$ is denoted by $\mathcal{E}'$. The space $\mathcal{E}$ can be embedded in $\mathcal{E}'$ via $\widehat\ :\mathcal{E}\to \mathcal{E}'$ defined by $x\mapsto \hat{x}$, where $\hat{x}(y)=\langle x,y\rangle\ (y\in \mathcal{E})$. A Hilbert $C^*$-module $\mathcal{E}$ is called self-dual if $\widehat{\mathcal{E}}=\mathcal{E}'$. We refer the reader to 
\cite{Lance,WO} for more details on the theory of Hilbert $C^*$-modules. 

Throughout this article, $S$ and $X$ stand for nonempty sets. We denote $C^*$-algebras by $\mathfrak{A}$ and $\mathfrak{B}$. By $\mathcal{Z}(\mathfrak{A})$ and $\mathfrak{A}^+$ we mean the center of the $C^*$-algebra $\mathfrak{A}$ and the set of all positive elements of $\mathfrak{A}$, respectively. When $\mathfrak{A}$ is unital, ${\rm Inv}(\mathfrak{A})$ stands for the set of all invertible elements of $\mathfrak{A}$. The $C^*$-algebra of all $n\times n$ matrices with entries in $\mathfrak{A}$ is presented by $\mathbb{M}_n(\mathfrak{A})$.

Aronszajn \cite{Aron} defined the concept of reproducing kernel Hilbert space (shortly, $RKHS$), and Schwartz \cite{119} developed the concept. This theory has many applications in integral equations, complex analysis, and so on; see \cite{12}. Indeed, an $RKHS$ $\mathcal{H}$ is a Hilbert space of $\mathbb{C}$-valued functions on a set $S$ such that, for all $s\in S$, the evaluation map $\delta_s:\mathcal{H}\to \mathbb{C}$ defined by $\delta_s(f)=f(s)$ is bounded. It follows from the Riesz representation theorem that, for every $s\in S$, there exists a unique element $k_s\in \mathcal{H}$ such that 
\[
\delta_s(f)=f(s)=\langle f,k_s\rangle\quad (f\in \mathcal{H}).
\]
Furthermore, the two-variable function $K:S\times S\to \mathbb{C}$ defined by $K(s,t)=k_t(s) (s,t\in S)$ is called a reproducing kernel for $\mathcal{H}$. A theorem due to Moore \cite[Theorem 2.14]{An} states that for a scalar-valued positive definite kernel, there is a unique $RKHS$ such that $K$ is its reproducing kernel. Indeed, there is a two-sided relation between scalar-valued positive definite kernels and $RKHS$s. For more information about reproducing kernel spaces we refer the interested reader to \cite{ALP, An, MAS} and references therein.

The Papadakis theorem \cite[Theorem 2.10]{An} shows that $\{f_s:s\in S\}$ is a Parseval frame for an $RKHS$ if and only if $K(x,y)=\sum_{s\in S}\overline{f_s(x)}f_s(y)$, where the series converges pointwise. In general, finding $k_s$ for every $s\in S$, and so $K$, is not easy, but the Papadakis theorem provides a useful benchmark. 

Although Hilbert $C^*$-modules generalize Hilbert spaces, some fundamental properties of Hilbert spaces are no longer valid in Hilbert $C^*$-modules in their full generality. For instance, not every bounded $\mathfrak{A}$-linear operator is adjointable. Thus in the theory of Hilbert $C^*$-modules, it is interesting to ask which results, similar to those for Hilbert spaces, can be proved probably under some conditions. Inspiring by some ideas in the Hilbert space setting \cite{An}, we extend some significant classical results to the setting of Hilbert $C^*$-modules.

The paper is organized as follows.\\
In the next section, we use some ideas of Szafraniec \cite{Szafraniec2} to give a modified definition of a reproducing kernel Hilbert $C^*$-module (shortly, $RKHC^*M$) due to Heo \cite{Heo} without using the condition of self-duality and discuss some related aspects. Such a lack of self-duality shows that our investigation is nontrivial and is not a straightforward generalization of the classical case of $RKHS$s. In the same section, the exterior tensor product of $RKHC^*M$s is investigated and an interpolation theorem is presented. Section 3 deals with left multipliers of $RKHC^*M$s. Under some mild conditions, we show that one can make a new $RKHC^*M$ by a left multiplier. In addition, we introduce the Berezin transform of an operator in the context of $RKHC^*M$s and construct a unital subalgebra of the unital $C^*$-algebra consisting of adjointable maps on an $RKHC^*M$ and show that it is closed with respect to a certain topology. In section 4, we extend the Papadakis theorem to the setting of $RKHC^*M$ and find some conditions, in order for the multiplication of two specific functions to be in the Papadakis $RKHC^*M$.


\section{A modified definition of $RKHC^*M$}

We denote by $\mathbb{F}(S,\mathfrak{A})$ the set of all $\mathfrak{A}$-valued functions on $S$. It is clear that $\mathbb{F}(S,\mathfrak{A})$ is a right $\mathfrak{A}$-module equipped with the ordinary operations. 

\begin{definition}[see \cite{Heo}]\label{msm}
By a kernel on $S$ we mean a map $K:S\times S\to \mathfrak{A}$. A kernel $K$ is called positive definite whenever the matrix $\big(K(s_i,s_j)\big)_{i,j=1}^n\in\mathbb{M}_n(\mathfrak{A})$ is positive or, equivalently,
\[
\sum_{i,j=1}^n a_i^*K(s_i,s_j)a_j\geq 0,
\]
for all $n\in \mathbb{N}$, $s_1, s_2, \ldots, s_n\in S$, and $a_1, a_2, \ldots, a_n\in \mathfrak{A}$. Then, $K(s,t) = K(t,s)^*$ for all $s, t\in S$. We say that a kernel $K$ is strictly positive whenever $\big(K(s_i,s_j)\big)_{i,j=1}^n$ is positive and invertible in $\mathbb{M}_n(\mathfrak{A})$.
\end{definition}

The following definition of an $RKHC^*M$ differs from that of \cite{Heo} and is inspired by \cite{Szafraniec2}.

\begin{definition}\label{mnm2}
A right $\mathfrak{A}$-submodule $\mathcal{E}$ of $\mathbb{F}(S,\mathfrak{A})$ is called a reproducing kernel Hilbert $C^*$-module if it satisfies the following conditions:
 \begin{itemize}
 \item[(i)]
 $\mathcal{E}$ is a Hilbert $C^*$-module over $\mathfrak{A}$.
 \item[(ii)]
 For every $s\in S$, there exists $k_s\in \mathcal{E}$ such that the evaluation map $\delta_s:\mathcal{E}\to \mathfrak{A}$ at $s\in S$ satisfies $\delta_s(f)=f(s)=\langle k_s,f\rangle$ for all $f\in \mathcal{E}$.
\item[(iii)]
The $\mathfrak{A}$-linear span of $\{k_s: s\in S\}$ is dense in $\mathcal{E}$.
 \end{itemize}
 The element $k_s$ is called the reproducing kernel for the point $s\in S$. 
 \end{definition}
 The corresponding reproducing kernel $K:S\times S\to \mathfrak{A}$ is given by $K(s,t)=\langle k_s, k_t\rangle$ for every $s,t\in S$; see \cite{MOS}. In what follows, we use the following theorems.
 \begin{theorem}\cite[proposition 3.1]{Heo}
 Let $S$, $\mathfrak{A}$, and $\mathcal{E}$ be as above. Then 
 \begin{itemize}
 \item[(i)]
 the kernel $K$ is positive definite;
 \item[(ii)]
 $K(s,s)\in \mathfrak{A}^+$ for every $s\in S$;
 \item[(iii)]
 $\|K(s,t)\|^2\leq\|K(s,s)\|\|K(t,t)\|$ for all $s,t\in S$.
 \end{itemize}
 \end{theorem} 
Now let $K:S\times S\to \mathfrak{A}$ be a positive definite kernel. For every $s\in S$, consider the function $k_s:S\to\mathfrak{A}$ by $k_s(t)=K(t,s)$. Assume that $\mathcal{E}_0$ is the right $\mathfrak{A}$-module of $\mathfrak{A}$-valued functions on $S$ generated by $\{k_s:s\in S\}$. Now, setting
\[
\left\langle \sum_{i=1}^mk_{s_i}a_i, \sum_{j=1}^nk_{t_j}b_j\right\rangle:=\sum_{i=1}^m\sum_{j=1}^na_i^*K(s_i,t_j)b_j,
\]
where $m,n\in \mathbb{N}$, $a_1, a_2, \ldots, a_n, b_1, b_2, \ldots, b_n\in \mathfrak{A}$, and $s_1, s_2, \ldots, s_m, t_1, t_2, \ldots, t_n\in S$, we make $\mathcal{E}_0$ into a pre-Hilbert $\mathfrak{A}$-module. Suppose that $\mathcal{E}$ denotes its completion. In addition, $\mathcal{E}$ can be considered as a Hilbert $\mathfrak{A}$-module of $\mathfrak{A}$-valued functions on $S$, and evidently $\langle k_s,k_t\rangle=K(s,t)$ for all $s,t\in S$, which means that $K$ is its reproducing kernel. The above construction is given by Heo \cite{Heo} and holds true for the definition \ref{mnm2}. It entails the following result.
 
\begin{theorem}\cite[Theorem 3.2]{Heo}\label{0121}
If $K:S\times S\to \mathfrak{A}$ is positive definite, then there exists a unique Hilbert $\mathfrak{A}$-module consisting of $\mathfrak{A}$-valued functions on $S$ such that $K$ is its reproducing kernel. 
\end{theorem}

Heo \cite{Heo} introduced the concept of an $RKHC^*M$ and transferred some of the classical theorems to the setting of $RKHC^*M$s. Regretfully, the requirement of self-duality in \cite{Heo} is too strong to hold in nontrivial examples. For example, it is known that finitely generated Hilbert $C^*$-modules and Hilbert $C^*$-modules over finite-dimensional $C^*$-algebras are self-dual; see \cite{FRA} and references therein.

Here, we explain why $RKHC^*M$s are almost never self-dual; see \cite{FRA}. For simplicity, we assume $S=\mathbb N$. Then $\xi=\sum_s k_s a_s$ can be viewed as a sequence $\xi=(a_s)_{s\in\mathbb N}$, where $a_s\in\mathfrak{A}$, $s\in\mathbb N$. The inner product on $\mathcal{E}$ is given by 
\[\langle \xi,\eta\rangle=\sum_{s,t\in\mathbb N}a_s^*K(s,t)b_t.\]
A sequence $\xi$ lies in $\mathcal{E}$ if and only if the series $\sum_{s,t\in\mathbb N}a_s^*K(s,t)a_t$ is norm-convergent. Let $f=(f_s)$ be a sequence such that all partial sums $\sum_{s,t\in\mathbb N}f_s^*K(s,t)f_t$ are uniformly bounded. Then the formula $f(\xi)=\sum_{s,t\in\mathbb N}f_s^*K(s,t)a_t$ is well-defined (i.e., the series is norm-convergent) and gives a bounded $\mathfrak{A}$-linear map from $\mathcal{E}$ to $\mathfrak{A}$, that is, $f\in\mathcal{E}'$. It is easy to see that the two conditions, norm convergence and uniform boundedness, are different in most cases. For simplicity, assume that $\mathfrak{A}$ is commutative, and so is of the form $C_0(X)$ for a locally compact Hausdorff space $X$. Then $K(s,t)$s are functions on $X$. Suppose that $\cup_{s\in\mathbb N}\operatorname{supp} K(s,s)$ is infinite. Then one can find functions $f_s$, $s\in\mathbb N$, on $X$ such that $\|K(s,s)f_s\|=1$ (where $\|\cdot\|$ is the sup-norm on $X$) and $f_sf_t=0$ when $s\neq t$. Then, obviously, $(f_s)_{s\in\mathbb N}\in\mathcal{E}'\setminus\mathcal{E}$. Note that since $K$ is positive definite, the condition that $\cup_{s\in\mathbb N}\operatorname{supp} K(s,s)$ is finite, implies that $\operatorname{supp}K(s,t)$ is finite as well.\\

While every closed subspace of a Hilbert space is orthogonally complemented, it is well known that submodules in Hilbert $C^*$-modules are often not orthogonally complemented. Furthermore, it is shown in \cite{MAG} that if there exists a full Hilbert $\mathfrak{A}$-module in
which every closed submodule is orthogonally complemented, then $\mathfrak{A}$ is $*$-isomorphic
to a $C^*$-algebra of (not necessarily all) compact operators on some Hilbert spaces; see \cite{FP} and references therein. The following theorem provides a class of orthogonally complemented submodules being $RKHC^*M$s. To achieve the next result, we need a lemma.

\begin{lemma}\label{mos2}
Let $\mathcal{E}$ be an $RKHC^*M$ on a set $S$ with the kernel $K$. Then every orthogonally complemented submodule $\mathcal{E}_0$ of $\mathcal{E}$ can be endowed with an $RKHC^*M$ on $S$.
\end{lemma}
\begin{proof}
Let $P: \mathcal{E}\to \mathcal{E}_0$ be the orthogonal projection onto $\mathcal{E}_0$. If $k_s$ is the reproducing kernel at the point $s$ in $\mathcal{E}$, then $P(k_s)$ evidently is the reproducing kernel for the point $s$ in $\mathcal{E}_0$ satisfying the conditions of Definition \ref{msm}. Note that
$$
f(s)=\langle k_s,f\rangle=\langle k_s,P(f)\rangle=\langle P(k_s),f\rangle, 
$$
for all $f\in\mathcal{E}_0$. Thus $\mathcal{E}_0$ is an $RKHC^*M$ with the reproducing kernel $K_0(s,t)=\langle P(k_s),P(k_t)\rangle$.
\end{proof}
\begin{theorem}
Suppose that $K(s_0,s_0)$ is invertible for some $s_0\in S$. Then the following properties hold:
\begin{itemize}
\item[(i)] The submodule $\mathcal{E}_0=\{f\in\mathcal{E}:f(s_0)=0\}$ is orthogonally complemented in $\mathcal{E}$.
\item[(ii)] $\mathcal{E}_0$ is the $RKHC^*M$ with the reproducing kernel 
\[K_0(s,t)=K(s,t)-K(s,s_0)K(s_0,s_0)^{-1}K(s_0,t).\]
\end{itemize}
\end{theorem} 
\begin{proof}
\begin{itemize}
\item[(i)] Note that $k_{s_0}$ satisfying $\langle k_{s_0},k_{s_0}\rangle=K(s_0,s_0)$ is invertible. Hence $P(f)=k_{s_0}\langle k_{s_0},k_{s_0}\rangle^{-1}\langle k_{s_0},f\rangle$ is the orthogonal projection onto the $\mathfrak{A}$-linear span of $k_{s_0}$. Then $\mathcal{E}=k_{s_0}\mathfrak{A}\oplus (k_{s_0}\mathfrak{A})^\perp$. The condition $f\perp k_{s_0}$ can be written as $0=\langle k_{s_0},f\rangle=f(s_0)$. Hence $(k_{s_0}\mathfrak{A})^\perp=\mathcal{E}_0$. 
\item[(ii)] The reproducing kernel for $\mathcal{E}_0$ is given by
\begin{eqnarray*}
\langle P(k_s), P(k_t)\rangle&=&\langle k_s,P(k_t)\rangle=\langle k_s,k_t-k_{s_0}\langle k_{s_0},k_{s_0}\rangle^{-1}\langle k_{s_0},k_t\rangle\rangle\\
&=&K(s,t)-K(s,s_0)K(s_0,s_0)^{-1}K(s_0,t)\\
&=&K_0(s,t).
\end{eqnarray*}
\end{itemize}
\end{proof}
The next result provides an interpolation theorem. We should notify that a finitely generated $\mathfrak{A}$-submodule is not necessarily closed. For example, when $\mathfrak{A}=C([0,1])=\mathcal{E}$, the submodule singly generated by the function $f(x)=x$ is not closed, and its closure is not finitely generated.

\begin{theorem}\label{msm2}
Let $\mathcal{E}$ be an $RKHC^*M$ on $S$ with the reproducing kernel $K$. Let $F=\{s_1, \ldots, s_n\}$ be a subset of distinct elements of $S$ such that the $C^*$-submodule generated by $k_{s_1},\ldots,k_{s_n}$ is closed in $\mathcal E$, and let $a_1, \ldots, a_n\in\mathfrak{A}$. Then there is a function $f\in \mathcal{E}$ of minimal norm such that $f(s_i)=a_i$ for all $1\leq i\leq n$ if and only if $(a_1, \ldots, a_n)^t$ is in the range of the matrix $\big(K(s_j,s_i)\big)\in \mathbb{M}_n(\mathfrak{A})$.
\end{theorem}
\begin{proof}
Assume that $P_F$ is the orthogonal projection onto the orthogonally complemented submodule $\mathcal{E}_F$ finitely generated by $\{k_s: s\in F\}$. So that $ \mathcal{E}= \mathcal{E}_F\oplus \mathcal{E}_F^\perp$, see \cite[Lemma 2.3.7]{Man}. Let $f\in \mathcal{E}$ and let $P_F(f)=\sum_{j=1}^nk_{s_j}b_j\in \mathcal{E}_F$, where $b_j$'s are in $\mathfrak{A}$. Then $P_F(f)(s)=f(s)$ for all $s\in F$, since $f(s)=\langle k_s,f\rangle=0$ for all $s\in F$ if and only if $f\in \mathcal{E}_F^\perp$.

If there is a function $f\in \mathcal{E}$ such that $f(s_i)=a_i$ for all $1\leq i\leq n$, then 
\[a_i=f(s_i)=P_F(f)(s_i)=\langle k_{s_i},P_F(f)\rangle=\left\langle k_{s_i},\sum_{j=1}^nk_{s_j}b_j\right\rangle=\sum_{j=1}^nK(s_i,s_j)b_j.\]
Thus $(a_1, \ldots, a_n)^t=\big(K(s_j,s_j)\big)(b_1,\ldots, b_n)^t$ is in the range of the matrix $\big(K(s_j,s_i)\big)$.\\
In addition, if $g\in\mathcal{E}$ interpolates these points, then $(g-f)(s)=0$ for all $s\in S$. Hence $g-f\in\mathcal{E}_F^\perp$, whence $g=f+h$ with $h\in \mathcal{E}_F^\perp$. Therefore,
\[\|P_F(f)\|=\|P_F(f+h)\|\leq \|f+h\|=\|g\|.\]
Thus, $P_F(f)$ is the unique function of minimum norm that interpolates these values.

Conversely, if $(a_1, \ldots, a_n)^t$ is in the range of $\big(K(s_j,s_i)\big)$ and $(a_1, \ldots, a_n)^t=\big(K(s_i,s_j)\big)(b_1,\ldots, b_n)^t$ for some $b_1, \ldots, b_n\in \mathfrak{A}$, then 
\[a_i=\sum_{j=1}^nK(s_i,s_j)b_j=\left\langle k_{s_i},\sum_{j=1}^nk_{s_j}b_j\right\rangle.\]
Putting $f:=\sum_{j=1}^nk_{s_j}b_j\in \mathcal{E}_F$, we get $f=P_F(f)$ and $a_i=\langle k_{s_i}, f\rangle=f(s_i)$ for $1\leq i\leq n$.
\end{proof}
\begin{remark}
If $\overline{a}=(a_1, \ldots, a_n)^t$ and $\overline{b}=(b_1,\ldots, b_n)^t$ are in the Hilbert $C^*$-module $\mathfrak{A}^n$ with its natural inner product (see \cite{Man}) and $(a_1, \ldots, a_n)^t = \big(K(s_i,s_j)\big) (b_1,\ldots, b_n)^t$, then we can choose $f$ such that $\|f\|=\|\langle \overline{a},\overline{b}\rangle\|^{1/2}$. In fact, if $f:=\sum_{j=1}^nk_{s_j}b_j$, then
\begin{align*}
\|f\|^2&=\left\|\left\langle \sum_{i=1}^nk_{s_i}b_i, \sum_{j=1}^nk_{s_j}b_j\right\rangle\right\|=\left\|\sum_{1\leq i,j\leq n}b_i^*K(s_i, s_j)b_j\right\|\\
&=\left\|\left\langle (b_i)_i, \left(\sum_{j=1}^nK(s_i,s_j)b_j\right)_i\right\rangle\right\|=\|\langle \overline{b},\overline{a}\rangle\|.
\end{align*}
If $K$ is strictly positive, then $\overline{b}$ is uniquely defined by $\overline{a}$. Thus, from the arguments at the first part of the proof of the above theorem, there is a unique $f\in \mathcal{E}_F$ satisfying the conditions of Theorem \ref{msm2}.
\end{remark}

Now, we investigate the exterior tensor product of $RKHC^*M$s. Let $K_1:X\times X\to \mathfrak{A}$ and $K_2:S\times S\to \mathfrak{B}$ be positive definite kernels on sets $X$ and $S$, respectively. It follows from Theorem \ref{0121} that there are $RKHC^*M$ $\mathcal{E}_1$ and $\mathcal{E}_2$ over $\mathfrak{A}$ and $\mathfrak{B}$ consisting of $\mathfrak{A}$-valued functions on $X$ and $\mathfrak{B}$-valued functions on $S$, respectively. We define $K:(X\times S)\times (X\times S)\to \mathfrak{A}\otimes_*\mathfrak{B}$, where $\otimes_*$ denotes a fixed $C^*$-tensor product with the $C^*$-cross-norm $\|\cdot\|_*$ by 
\[
K((x,s),(y,t))=K_1(x,y)\otimes K_2(s,t),\qquad (x,s), (y,t)\in X\times S.
\] 
Let $\xi_i=\sum_{k=1}^n a^k_i\otimes b^k_i\in\mathfrak{A}\otimes_{alg}\mathfrak{B}$, $i=1,\ldots,m$. Then
\begin{align*}
\sum_{i,j=1}^n&\xi_i^*K((x_i,s_i),(x_j,s_j)) \xi_j\cr
&=\sum_{k,l=1}^n\left(\sum_{i,j=1}^n (a_i^k)^*K_1(x_i,x_j)a_j^l\right)\otimes \left(\sum_{i,j=1}^n (b_i^k)^*K_2(s_i,s_j)b_j^l\right).
\end{align*}
Set 
$$
\alpha_{kl}:=\sum_{i,j=1}^n (a_i^k)^*K_1(x_i,x_j)a_j^l\in\mathfrak{A},\quad \beta_{kl}:=\sum_{i,j=1}^n (b_i^k)^*K_2(s_i,s_j)b_j^l\in\mathfrak{B}. 
$$
Then the matrices $(\alpha_{kl})_{k,l=1}^n$ and $(\beta_{kl})_{k,l=1}^n$ are positive elements of $M_n(\mathfrak{A})$ and of $M_n(\mathfrak{B})$, respectively. Then, by Lemma 4.3 of \cite{Lance}, 
\begin{equation}\label{015}
\sum_{i,j=1}^n\xi_i^*K((x_i,s_i),(x_j,s_j)) \xi_j=\sum_{k,l=1}^n\alpha_{kl}\otimes \beta_{kl}\geq 0.
\end{equation}

Since the set of all positive elements in a $C^*$-algebra is closed, we conclude that \eqref{015} is also valid for every choice of elements $\xi$ in $\mathfrak{A}\otimes_*\mathfrak{B}$. Thus $K$ is a positive definite kernel. Again, in virtue of Theorem \ref{0121}, there exists a Hilbert $\mathfrak{A}\otimes_*\mathfrak{B}$-module $\mathcal{E}$ of $\mathfrak{A}\otimes_*\mathfrak{B}$-valued functions on $X\times S$ such that $K$ is its reproducing kernel. Now this question raises: What relations are there between $\mathcal{\mathcal{E}}$, $\mathcal{E}_1$, and $\mathcal{E}_2$, where $\mathcal{E}_1$ and $\mathcal{E}_2$ are $RKHC^*M$s with kernels $K_1$ and $K_2$, respectively? 

Recall that the exterior tensor product $\mathcal{E}_1\otimes\mathcal{E}_2$ of Hilbert $C^*$-modules $\mathcal{E}_1$ over $\mathfrak{A}$ and $\mathcal{E}_2$ over $\mathfrak{B}$ is defined as the Hilbert $C^*$-module over $\mathfrak{A}\otimes_*\mathfrak{B}$ obtained by completion of $\mathcal{E}_1\otimes_{alg}\mathcal{E}_2$ with respect to the norm 
$$
\|u\|^2=\Bigl\|\sum_{i,j=1}^n\langle f_i,f_j\rangle\otimes\langle g_i,g_j\rangle\Bigr\|_*, 
$$
where $u=\sum_{i=1}^n f_i\otimes g_i\in\mathcal{E}_1\otimes_{alg}\mathcal{E}_2$; see \cite{Lance}.

Set $k_y^1(x):=K_1(x,y)$, $k_t^2(s):=K_2(s,t)$, and $k_{(y,t)}(x,s):=K((x,s),(y,t))$. Clearly, $k_{(y,t)}(x,s)=k^1_y(x)\otimes k^2_t(s)$. By the assumption, the $\mathfrak{A}$-linear spans of $\{k_x^1: x\in X\}$, $\{k_s^2: s\in S\}$, and $\{k_{(x,s)}: (x,s)\in X\times S\}$ are dense in $\mathcal{E}_1$, $\mathcal{E}_2$, and $\mathcal{E}$, respectively.

We claim that $\mathcal{E}_1\otimes \mathcal{E}_2$ is unitarily equivalent to $\mathcal{E}$. 
Let $f_i=\sum_{x\in X} k^1_x a^i_x$ and $g_i=\sum_{s\in S} k^2_s b_s^i$, where $a^i_x\in\mathfrak{A}$, $b^i_s\in\mathfrak{B}$, and both sums have a finite number of nonzero summands.
Set 
\begin{equation}\label{Phi}
\Phi\left(\sum_{i=1}^n\sum_{x\in X,s\in S} k_x^1 a_x^i\otimes k^2_s b_s^i\right):=\sum_{x\in X,s\in S}k_{(x,s)}\cdot\sum_{i=1}^n a_x^i\otimes b_s^i.
\end{equation}

For $u\in\mathcal{E}_1\otimes_{alg}\mathcal{E}_2$, define $\hat u\in\mathbb F(X\times S,\mathfrak{A}\otimes_*\mathfrak{B})$ by
\[
\hat{u}(x,s)=\langle k_x^1\otimes k_s^2,u\rangle, \qquad (x,s)\in X\times S.
\]
Let $u=\sum_{i=1}^n f_i\otimes g_i$, where $f_i=\sum_{x\in X} k^1_x a^i_x$, $g_i=\sum_{s\in S} k^2_s b_s^i$, where $a^i_x\in\mathfrak{A}$, $b^i_s\in\mathfrak{B}$, and both sums have a finite number of nonzero summands. Then
\begin{eqnarray*}
\hat u(y,t)&=&\sum_{i=1}^n\sum_{x\in X,s\in S}k^1_x(y)a_x^i\otimes k^2_s(t)b_s^i\\
&=&\sum_{x\in X,s\in S}k_{(x,s)}(y,t)\sum_{i=1}^n
a_x^i\otimes b_s^i
=\Phi(u)(y,t), 
\end{eqnarray*}
which shows that the map $\Phi$, defined in \eqref{Phi}, is well-defined and that $\hat u\in\mathcal{E}$. 
 
It is clear that $\Phi$ is an isometry between dense subspaces of $\mathcal{E}_1\otimes\mathcal{E}_2$ and of $\mathcal{E}$, hence it extends to a surjective isometry $\Phi:\mathcal{E}_1\otimes\mathcal{E}_2\to\mathcal{E}$. 

\begin{definition}
Suppose that $K_1:X\times X\to \mathfrak{A}$ and $K_2:S\times S\to\mathfrak{B}$ are kernels. We call the map $K:(X\times S)\times(X\times S)\to\mathfrak{A}\otimes_*\mathfrak{B}$ defined by 
\[
K((x,s),(y,t))=K_1(x,s)\otimes K_2(y,t),\qquad (x,s), (y,t)\in X\times S
\] 
the tensor product of the kernels $K_1$ and $K_2$ and denote it by $K_1\otimes K_2$.
\end{definition}
We summarize the above results in the following theorem.
\begin{theorem}
Let $K_1$ and $K_2$ be positive definite kernels and let $\mathcal{E}_1$ and $\mathcal{E}_2$ be their associated Hilbert $C^*$-modules. Then $K_1\otimes K_2$ is a positive definite kernel, and its associated Hilbert $C^*$-module is unitarily equivalent to the exterior tensor product of $\mathcal{E}_1$ and $\mathcal{E}_2$. 
\end{theorem}

\section{Left multipliers of $RKHC^*M$s}

If $F_1$ and $F_2$ are submodules of $\mathbb{F}(S,\mathfrak{A})$, then a function $f\in \mathbb{F}(S,\mathfrak{A})$ for which $fF_1\subseteq F_2$ is called a left multiplier of $F_1$ into $F_2$. Note that 
 \[
 fF_1=\{fh:h\in F_1\},
 \]
where $fh$ is the multiplication of $f$ and $h$. The set of all left multipliers of $F_1$ into $F_2$ is denoted by $\mathcal{M}(F_1,F_2)$. Clearly, $\mathcal{M}(F_1,F_2)$ is a linear space. Moreover, $\mathcal{M}(F)$ stands for $\mathcal{M}(F,F)$ being an algebra. For every $f\in \mathcal{M}(F_1,F_2)$, there is a linear map $M_f:F_1\to F_2$ that is defined by $M_f(h)=fh$ for all $h\in F_1$.\\	

The following lemma is a useful property of $RKHC^*M$s.

\begin{lemma}\label{most}
Suppose that $\mathcal{E}$ is an $RKHC^*M$ on a set $S$ with the kernel $K$. If a sequence $(h_n)$ in $\mathcal{E}$ converges to $h$, then $\lim_nh_n(s)=h(s)$ for each $s\in S$.
\end{lemma}
\begin{proof}
It is easily concluded from
\[\|h_n(s)-h(s)\|=\|\langle k_s,h_n\rangle - \langle k_s,h\rangle\|\leq \|k_s\|\,\|h_n-h\|.\]
\end{proof}

Let $\mathcal{E}$ be an $RKHC^*M$ on $S$ and let $g:S\to \mathfrak{A}$ be a function. Put \[\mathcal{E}_g=\{gh: h\in \mathcal{E}\}.\]
In the next theorem, we endow the right $\mathfrak{A}$-submodule $\mathcal{E}_g$ of $\mathbb{F}(S,\mathfrak{A})$ with an $RKHC^*M$ structure.
\begin{theorem}
Suppose that $\mathcal{E}$ is an $RKHC^*M$ on a set $S$ with the kernel $K$ and that $g:S\to \mathfrak{A}$ is an arbitrary function. Then the following assertions hold:
\begin{itemize}
\item[(i)] $\mathcal{E}_0:=\{h\in \mathcal{E}: gh=0\}$ is closed. 
\item[(ii)] If $\mathcal{E}_0$ is orthogonally complemented, then $\mathcal{E}_g$ is an $RKHC^*M$ with kernel $K^{'}(s,t)=g(s)K(s,t)g(t)^*$. 
\end{itemize} 
\end{theorem}
\begin{proof}
\begin{itemize}
\item[(i)] It follows from Lemma \ref{most} that $\mathcal{E}_0$ is closed. 
\item[(ii)] It follows from the assumption that $\mathcal{E}=\mathcal{E}_0\oplus\mathcal{E}_0^\perp$. Therefore 
\[
\mathcal{E}_g=\{g\tilde{h}+gh^{\#}:\tilde{h}\in \mathcal{E}_0, h^{\#}\in \mathcal{E}_0^\perp\}=\{gh:h\in \mathcal{E}_0^\perp\}.	
\]
We define an $\mathfrak{A}$-valued inner product on $\mathcal{E}_g$ by $\langle gh_1,gh_2\rangle=\langle h_1,h_2\rangle$ for all $h_1, h_2\in \mathcal{E}_0^\perp$. This is well-defined, since if $gh=gh^{'}$ for $h, h'\in \mathcal{E}_0^\perp$, then $h-h^{'}\in \mathcal{E}_0\cap \mathcal{E}_0^\perp=\{0\}$. From the inner product on $\mathcal{E}_g$, it is clear that $\varphi_g:\mathcal{E}_0^\perp\to \mathcal{E}_g$ by $\varphi_g(h)=gh$ is a surjective linear	 isometry. Hence $\mathcal{E}_g$ is a Hilbert $C^*$-module isomorphic with $\mathcal{E}_0^\perp$. Thus the reproducing kernel structure of $\mathcal{E}_0^\perp$ constructed in Lemma \ref{mos2} can be transferred onto $\mathcal{E}_g$. More precisely, for each $h\in \mathcal{E}_0^\perp$, we have
\begin{align*}
(gh)(s)&=g(s)h(s)=g(s)\langle k_s,h\rangle= g(s)\langle k_s^{\#}, h\rangle= g(s)\langle gk_s^{\#}, gh\rangle\\
&=g(s)\langle gk_s, gh\rangle=\langle gk_sg(s)^*,gh\rangle
\end{align*}
for some $k_s=\tilde{k_s}+k_s^{\#}\in \mathcal{E}_0\oplus\mathcal{E}_0^\perp$. Hence the evaluation map $\delta_s$ can be represented by $\langle gk_sg(s)^*,\cdot\rangle$ with $ k_s^{'}=gk_sg(s)^*\in \mathcal{E}_g$. In addition, the corresponding reproducing kernel is
\begin{align*}
K^{'}(s,t)&=\langle k^{'}_s,k^{'}_t\rangle=\langle gk_sg(s)^*,gk_tg(t)^*\rangle= g(s)\langle gk_s, gk_t\rangle g(t)^*\\
&=g(s)\langle gk_s^{\#}, gk_t^{\#}\rangle g(t)^*=g(s)\langle k_s^{\#}, k_t^{\#}\rangle g(t)^*+0\\
&=g(s)\langle k_s^{\#}, k_t^{\#}\rangle g(t)^*+g(s)\tilde{k_t}(s) g(t)^*\quad\qquad ({\rm since~} g\tilde{k_t}=0)\\
&=g(s)\langle k_s^{\#}, k_t^{\#}\rangle g(t)^*+g(s)\langle \tilde{k_s},\tilde{k_t}\rangle=g(s)\langle k_s, k_t\rangle g(t)^*\\
&=g(s)K(s,t)g(t)^*
\end{align*}
for every $s,t\in S$.
\end{itemize} 
\end{proof}
In the following theorem, $k_s^1$ and $k_s^2$ are the reproducing kernels at the point $s\in S$ for $RKHC^*M$s $\mathcal{E}_1$ and $\mathcal{E}_2$, respectively.
\begin{proposition}
Let $\mathcal{E}_1$ and $\mathcal{E}_2$ be $RKHC^*M$s on a nonempty set $S$. If $f\in \mathcal{M}(\mathcal{E}_1,\mathcal{E}_2)$, then $M_f\in L(\mathcal{E}_1,\mathcal{E}_2)$ and $M_f^*(k_s^2)=k_s^1f(s)^*$ for all $s\in S$.
\end{proposition}
\begin{proof}
For every $h\in \mathcal{E}_1$, $s_1, \ldots, s_n\in S$, and $a_1, \ldots, a_n\in\mathfrak{A}$, we have 
\begin{align*}
\left\langle \sum_{i=1}^nk_{s_i}^2a_i, M_f(h)\right\rangle &= \sum_{i=1}^na_i^*\langle k_{s_i}^2, fh\rangle=\sum_{i=1}^na_i^* f(s_i)h(s_i) \\
&=\sum_{i=1}^na_i^*f(s_i)\langle k_{s_i}^1,h\rangle=\left\langle \sum_{i=1}^nk_{s_i}^1f(s_i)^*a_i, h\right\rangle.
\end{align*}
Hence $M_f^*\left(\sum_{i=1}^nk_{s_i}^2a_i\right)=\sum_{i=1}^nk_{s_i}^1f(s_i)^*a_i$. In particular, $M_f^*(k_s^2)=k_s^1f(s)^*$ for all $s\in S$.
\end{proof}

Thus, if $f\in \mathcal{M}(\mathcal{E})$ and $\mathfrak{A}$ is a unital $C^*$-algebra, then
\begin{align*}
f(s)=\langle k_s, fk_s\rangle K(s,s)^{-1}=\langle k_s, M_f(k_s)\rangle K(s,s)^{-1},
\end{align*}
for every point $s\in S$ for which $K(s,s)\in {\rm Inv}(\mathfrak{A})$. Thus, we can present the following definition in the same manner as in the classical case \cite{An} and transfer some known facts in the theory of $RKHS$s to context of $RKHC^*M$.
\begin{definition}
Let $\mathcal{E}$ be an $RKHC^*M$ on $S$ over a unital $C^*$-algebra $\mathfrak{A}$. Let $K$ be its associated kernel and let $T\in L(\mathcal{E})$ be arbitrary. Then the function 
\[
B_T:\{s\in S: K(s,s) {\rm\ is\ invertible}\}\longrightarrow \mathfrak{A}
\]
 defined by $B_T(s)=\langle k_s, T(k_s)\rangle K(s,s)^{-1}$ is called the Berezin transform of $T$ associated by $\mathfrak{A}$. 
\end{definition}
\begin{theorem}
Let $\mathcal{E}$ be an $RKHC^*M$ on $S$ with the reproducing kernel $K$ over a unital $C^*$-algebra $\mathfrak{A}$. Let 
\[L=\{M_f:f\in \mathcal{M}(\mathcal{E}) {\rm~ and~} f(s)=0 {\rm~whenever~} K(s,s) {\rm~is~not~invertible}\}.\]
Then $L$ is a unital subalgebra of $L(\mathcal{E})$. 

Furthermore, if $\{M_{f_\alpha}\}_{\alpha\in I}$ is a net in $L$ such that $\langle M_{f_\alpha} h_1,h_2\rangle \to \langle Th_1,h_2\rangle\,\, (h_1,h_2\in\mathcal{E})$ for some $T\in L(\mathcal{E})$, then $T=M_f$ for some $f\in \mathbb{F}(S,\mathfrak{A})$. 
\end{theorem}
\begin{proof}
Since 
\[\
\lambda M_f+M_g=M_{\alpha f+g},\ M_f\circ M_g=M_{fg}\quad (f,g\in \mathcal{M}(\mathcal{E}), \lambda\in \mathbb{C}),
\]
$L$ is an algebra. Moreover, $M_1$ is the unit of $L$, where $1\in \mathcal{M}(\mathcal{E})$ is the constant function onto the unit of $\mathfrak{A}$. 

Next, we show that $T=M_f$ for some $f\in \mathcal{M}(\mathcal{E})$. We have 
\begin{align*}
\lim_\alpha f_\alpha(s)&=\lim_\alpha\langle k_s, M_{f_\alpha}(k_s)\rangle K(s,s)^{-1}=\langle k_s, T(k_s)\rangle K(s,s)^{-1}=B_T(s),
\end{align*}
for every $s\in S$ for which $K(s,s)$ is invertible. Set $f(s):=B_T(s)$ whenever $K(s,s)$ is invertible and $f(s):=0$ otherwise. To complete the proof, we shall show that $T=M_f$. We have
\begin{align*}
\left\langle \sum_{i=1}^nk_{s_i}a_i, Th\right\rangle &=\lim_{\alpha}\left\langle \sum_{i=1}^nk_{s_i}a_i, M_{f_\alpha}h\right\rangle=\lim_{\alpha}\left\langle \sum_{i=1}^nk_{s_i}a_i, f_\alpha h\right\rangle\\
&=\lim_{\alpha}\sum_{i=1}^na_i^* f_\alpha(s_i)h(s_i)=\sum_{i=1}^na_i^*f(s_i)h(s_i)\\
&=\sum_{i=1}^na_i^*\langle k_{s_i}, fh\rangle=\left\langle \sum_{i=1}^nk_{s_i}a_i, M_fh\right\rangle,
\end{align*}
for every $h\in \mathcal{E}$, $s_1, \ldots, s_n\in S$, and $a_1, \ldots, a_n\in\mathfrak{A}$. Thus $T=M_f$.
\end{proof}

\section{Papadakis theorem for $RKHC^*M$s}

We recall the following definitions from \cite{frame}.
\begin{definition}
Let $J$ be an arbitrary subset of $\mathbb{N}$ and let $\mathcal{E}$ be a Hilbert $C^*$-module over a unital $C^*$-algebra $\mathfrak{A}$. A sequence $(x_j)_{j\in J}$ in $\mathcal{E}$ is said to be a (standard) frame if there are real numbers $C, D>0$ such that 
\begin{align}\label{0}
C\langle x,x\rangle\leq \sum_{j\in J}\langle x,x_j\rangle\langle x_j, x\rangle\leq D\langle x,x\rangle
\end{align} 
for every $x\in \mathcal{E}$ in which the sum in the middle of inequality \eqref{0} converges in norm.
\end{definition} 
The sharp numbers (i.e., maximal for $C$ and minimal for $D$) are called frame bounds. A frame $\{x_j: j\in J\}$ is said to be a tight frame if $C=D$, and normalized if $C=D=1$. Therefore, a set $\{x_j: j\in J\}$ is a normalized tight frame whenever the equality 
\begin{equation} \label{1}
\langle x,x\rangle=\sum_{j\in J} \langle x,x_j\rangle\langle x_j,x\rangle
\end{equation}
is valid for every $x\in \mathcal{E}$. 

Now, we extend the Papadakis theorem to $RKHC^*M$s. 
\begin{theorem}
Let $\mathcal{E}$ be an $RKHC^*M$ on a set $S$ over a unital $C^*$-algebra $\mathfrak{A}$ and let $K$ be its corresponding reproducing kernel. Then $\{f_j: j\in J\}\subseteq \mathcal{E}$ is a normalized tight frame for $\mathcal{E}$ if and only if 
\begin{equation}\label{2}
K(s,t)=\sum_{j\in J}f_j(s)^*f_j(t)\quad (s, t\in S),
\end{equation}
where the sum is convergent in norm.
\end{theorem}

\begin{proof}
Suppose that $\{f_j: j\in J\}$ is a normalized tight frame for $\mathcal{E}$. It follows from \eqref{1} that
\begin{equation*}
\langle f,f\rangle=\sum_{j\in J}\langle f,f_j\rangle\langle f_j,f\rangle,
\end{equation*}
for every $f\in \mathcal{E}$. Therefore, by the polarization identity, we can write 
\begin{align*}
K(s,t)&=\langle k_s, k_t\rangle=\frac{1}{4}\sum_{k=0}^3 i^k\langle k_t+i^kk_s, k_t+i^kk_s\rangle\\
&=\frac{1}{4}\sum_{k=0}^3 i^k\sum_{j\in J}\langle k_t+i^kk_s,f_j\rangle\langle f_j, k_t+i^kk_s\rangle\\
&= \frac{1}{4}\sum_{k=0}^3 i^k\sum_{j\in J}(f_j(t)+i^kf_j(s))^*(f_j(t)+i^kf_j(s))\\
&=\sum_{j\in J}\frac{1}{4}\sum_{k=0}^3i^k\langle f_j(t)+i^kf_j(s),f_j(t)+i^kf_j(s)\rangle\\
&= \sum_{j\in J}\langle f_j(s), f_j(t)\rangle=\sum_{j\in J}f_j(s)^*f_j(t)
\end{align*}
for all $s,t\in S$. 

Conversely, let \eqref{2} hold for some family $\{f_j: j\in J\}$ and let the sum in \eqref{2} converge in the norm topology. Then
\begin{align*}
\langle k_s, k_s\rangle=K(s,s)=\sum_{j\in J}f_j(s)^*f_j(s)=\sum_{j\in J}\langle k_s,f_j\rangle\langle f_j, k_s\rangle
\end{align*}
for every $s\in S$. Hence, by the density of $\mathfrak{A}$-linear span of $\{k_s: s\in S\}$ in $\mathcal{E}$ and the joint continuity of inner product, we derive 
\begin{align*}
\langle f,f\rangle=\sum_{j\in J}\langle f,f_j\rangle\langle f_j,f\rangle
\end{align*} 
for all $f\in \mathcal{E}$. It follows from \eqref{1} that $\{f_j: j\in J\}$ is a normalized tight frame. 
\end{proof}

As we already mentioned in the introduction, $RKHC^*M$s are rarely self-dual. Recall that $\mathcal{E}'$ denotes the dual module of $\mathcal{E}$. 

\begin{lemma}\label{vlad1}
Elements of $\mathcal{E}'$ can be thought of as functions on $S$, that is, there is an inclusion $\mathcal{E}'\subset\mathbb F(S,\mathfrak{A})$ that extends the inclusion $\mathcal{E}\subset\mathbb F(S,\mathfrak{A})$.
\end{lemma}
\begin{proof}
Suppose that $F\in\mathcal{E}'$. Set $F(s):=F(k_s)$. This gives us a map $\mathcal{E}'\to\mathbb F(S,\mathfrak{A})$. To show that this map is faithful, suppose that $F(s)=0$ for any $s\in S$. Then $F$ vanishes on a dense subset of $\mathcal{E}$, and hence is zero.
\end{proof}

It is clear that if $K$ is a kernel of the form $\eqref{2}$, then it is a positive definite kernel; see \cite{GMX} for the Kolmogorov decomposition at the setting of Hilbert $C^*$-modules. Hence, employing Theorem \ref{0121}, there exists a Hilbert $\mathfrak{A}$-module consisting of $\mathfrak{A}$-valued functions on $S$ such that $K$ is its reproducing kernel. This is a motivation for the following definition.

\begin{definition}
Let $K:S\times S\to \mathfrak{A}$ be the positive definite kernel defined by
\[
 K(s,t)=\sum_{\alpha\in I} e_\alpha(s)^*e_\alpha(t) \quad (s,t\in S),
\]
where $\{e_\alpha\}_{\alpha\in I}$ is a family in $\mathbb{F}(S,\mathfrak{A})$ with the property that $\sum_{\alpha\in I} e_\alpha(s)^*e_\alpha(s)$ converges in $\mathfrak{A}$. Then $K$ is called the Papadakis kernel, and the Hilbert $\mathfrak{A}$-module consisting of $\mathfrak{A}$-valued functions on $S$, given by Theorem \ref{0121}, is called the Papadakis Hilbert $\mathfrak{A}$-module.
\end{definition}
 
Let $K$ be a Papadakis kernel on $S$ for some family $\{e_\alpha\}_{\alpha\in I}\subseteq \mathbb{F}(S,\mathfrak{A})$ and let $\mathcal{E}$ be the associated Papadakis Hilbert $C^*$-module. The following theorem shows that the multiplication of an element of $\mathbb{F}(S,\mathfrak{A})$ satisfying suitable conditions and that $e_\alpha \ (\alpha\in I)$ is an element of $\mathcal{E}$. Note that $e_\alpha^*:S\to\mathfrak{A}$ is defined by $e_\alpha^*(s)=e_\alpha(s)^*$ for all $\alpha\in I$ and $s\in S$. To achieve our next result, we mimic some ideas of \cite{zaf}.

\begin{definition}
A subset $P$ of $S$ is said to be a set of uniqueness of $\mathcal{E}\subseteq\mathbb{F}(S,\mathfrak{A})$ if the $\mathfrak{A}$-linear span of $k_p$, $p\in P$, is dense in $\mathcal{E}$. In this case, we write $P\in \mathbb{U}(\mathcal{E})$.
\end{definition}
Note that if $f,g\in\mathcal{E}$ with $f(p)=g(p)$ for any $p\in P$, then $f=\sum_{p\in P} k_p a_p$ and $g=\sum_{p\in P} k_p b_p$, where $a_p,b_p\in\mathfrak{A}$. Then $f=g$.

\begin{theorem}
Let $K$ be the Papadakis kernel for some family $\{e_\alpha\}_{\alpha\in I}\subseteq \mathbb{F}(S,\mathfrak{A})$ and let $\mathcal{E}$ be the associated Papadakis Hilbert $C^*$-module. Let $e_\alpha\in \mathcal{M}(\mathcal{E}',\mathcal{E})$ and let $X\in \mathbb{U}(\mathcal{E})$. Assume that $\psi: X\to \mathcal{Z}(\mathfrak{A})$ is a function and that $c>0$ is such that 
\begin{align}\label{0202}
\sum_{i,j=1}^n a_i^*K(x_i,x_j)(c^2-\psi(x_i)^*\psi(x_j))a_j\geq 0,
\end{align}
for all $n\in \mathbb{N}, x_1, x_2, \ldots, x_n\in X$, and $a_1, a_2,\ldots, a_n\in \mathfrak{A}$. Then for every $\alpha\in I$, there is a unique function $\varphi_\alpha\in \mathcal{E}$ such that 
\begin{align*}
\varphi_\alpha(x)&=e_\alpha(x)\psi(x), \qquad x\in X
\end{align*}
or, equivalently,
\begin{align*}
\varphi_\alpha=e_\alpha\psi
\end{align*}
and
\[
e_\alpha \varphi_\beta =e_\beta \varphi_\alpha,
\]
for all $\alpha, \beta\in I$. Furthermore, if ${\rm ran}(e_\alpha)\subseteq \mathcal{Z}(\mathfrak{A})$ and $K(s,s)$ is invertible for every $s\in S$, then
\begin{align*}
|\varphi_\alpha(s)|&\leq c|e_\alpha(s)|, (s\in S).
\end{align*}
\end{theorem}
\begin{proof}
Inequality \eqref{0202} can be restated as follows:
\begin{align}\label{0123}
c^2\left\langle \sum_{i=1}^nk_{x_i}a_i,\sum_{j=1}^nk_{x_j}a_j\right\rangle &\geq 
\sum_{\alpha\in I}\left(\sum_{i=1}^n e_\alpha(x_i)a_i\psi(x_i)\right)^*\left(\sum_{j=1}^n e_\alpha(x_j)a_j\psi(x_j)\right)\cr
&=\sum_{\alpha\in I}\left|\sum_{i=1}^n e_\alpha(x_i)a_i\psi(x_i)\right|^2
\end{align}
for every $n\in \mathbb{N}, x_1, x_2, \ldots, x_n\in X$ and $a_1, a_2, \ldots, a_n\in \mathfrak{A}$. Put 
\[
D=\left\{\sum_{i=1}^nk_{x_i}a_i:n\in\mathbb{N},x_1, x_2, \ldots, x_n\in X, a_1, a_2, \ldots, a_n\in \mathfrak{A}\right\}.
\]
For every $\alpha\in I$, we define $\varphi_\alpha:D\to \mathfrak{A}$ by 
\begin{align}\label{023}
\varphi_\alpha\left(\sum_{i=1}^nk_{x_i}a_i\right)=\sum_{i=1}^n e_\alpha(x_i)a_i\psi(x_i),
\end{align}
where $n\in \mathbb{N}, x_1, x_2, \ldots, x_n\in X$, and $a_1, a_2,\ldots, a_n\in \mathfrak{A}$. Set $b:=\sum_{i=1}^n k_{x_i}a_i$. From \eqref{0123}, we conclude that 
\begin{equation*}
c^2\langle b,b\rangle\geq \varphi_\alpha(b)^*\varphi_\alpha(b)
\end{equation*}
for every $\alpha\in I$. Hence $\varphi_\alpha$ is a well-defined bounded $\mathfrak{A}$-linear map. Since $X\in \mathbb{U}(\mathcal{E})$, the set $D$ is dense in $\mathcal{E}$. Hence, we can extend $\varphi_\alpha$ to $\mathcal{E}$. For simplicity, we denote it by the same $\varphi_\alpha$, so that $\varphi_\alpha\in\mathcal{E}'$. 
From \eqref{0123}, we reach
\begin{align}\label{14725}
c^2|g|^2\geqslant \sum_{\alpha\in I}|\varphi_\alpha(g)|^2 \quad (g\in \mathcal{E}).
\end{align}
Utilizing Lemma \ref{vlad1} and \eqref{023}, we arrive at
\begin{align*}
\varphi_\alpha(x)=\varphi_\alpha(k_x)=e_\alpha(x)\psi(x) \quad (x\in X, \alpha\in I).
\end{align*}
Then
\begin{align}\label{147}
e_\alpha(x)\varphi_\beta(x)&=\varphi_\beta(x)e_\alpha(x) \quad (x\in X,\ \alpha,\beta\in I).
\end{align}
Since $e_\alpha\in \mathcal{M}(\mathcal{E})$ and $X\in \mathbb{U}(\mathcal{E})$, from \eqref{147}, we infer that 
\begin{align*}
e_\alpha(s)\varphi_\beta(s)&=\varphi_\beta(s)e_\alpha(s), \quad (s\in S).
\end{align*}
Now, fix $\alpha\in I$ and $s\in S$. Putting $g=k_s$ in \eqref{14725}, we arrive at 
\begin{align*}
c^2K(s,s)\geq \sum_{\beta\in I}\varphi_\beta(s)^*\varphi_\beta(s).
\end{align*}
 Since ${\rm ran}(e_\alpha)\subseteq \mathcal{Z}(\mathfrak{A})$, we have 
\begin{align*}
c^2K(s,s)e_\alpha(s)^*e_\alpha(s)&\geq\sum_{\beta \in I}(e_\alpha(s)\varphi_\beta(s))^*e_\alpha(s)\varphi_\beta(s)\\
&=\sum_{\beta\in I}e_\beta(s)^*\varphi_\alpha(s)^*e_\beta(s)\varphi_\alpha(s)\\
&=|\varphi_\alpha(s)|^2K(s,s).
\end{align*}
Now, the invertibility of $K(s,s)$ entails that $|\varphi_\alpha(s)|\leq c |e_\alpha(s)|$. 
\end{proof}

\medskip

\noindent \textit{Conflict of Interest Statement.} On behalf of all authors, the corresponding author states that there is no conflict of interest.\\

\noindent\textit{Data Availability Statement.} Data sharing not applicable to this article as no datasets were generated or analysed during the current study.

\end{document}